\documentclass[preprint, 3p]{elsarticle}

\usepackage{graphicx}

\usepackage{amsmath, amssymb, amsfonts}
\makeatletter
\DeclareRobustCommand{\qed}{%
  \ifmmode 
  \else \leavevmode\unskip\penalty9999 \hbox{}\nobreak\hfill
  \fi
  \quad\hbox{\qedsymbol}}
\newcommand{\openbox}{\leavevmode
  \hbox to.77778em{%
  \hfil\vrule
  \vbox to.675em{\hrule width.6em\vfil\hrule}%
  \vrule\hfil}}
\newcommand{\qedsymbol}{\openbox}
\newenvironment{proof}[1][\proofname]{\par
  \normalfont
  \topsep6\p@\@plus6\p@ \trivlist
  \item[\hskip\labelsep\itshape
    #1.]\ignorespaces
}{%
  \qed\endtrivlist
}
\newcommand{\proofname}{Proof}
\makeatother
\renewcommand\qedsymbol{$\blacksquare$}

\usepackage{mathtools}
\usepackage{times}

\graphicspath{'./Figs'}
\usepackage{dsfont}
\usepackage{color}
\usepackage{empheq}
\usepackage{caption}
\usepackage{subcaption}

\usepackage{blindtext} 
\usepackage{diffcoeff}

\DeclareMathOperator*{\argmin}{arg\,min}
\diffdef { p }
{
op-symbol = \partial ,
left-delim = \left .,
right-delim = \right | ,
subscr-nudge = +4 mu
}

\newtheorem{theorem}{Theorem}

\newtheorem{lemma}{Lemma}

\newtheorem{definition}{Definition}
\newtheorem{assumption}{Assumption}
\newtheorem{remark}{Remark}

\newcommand{\B}{\mathbb{B}}
\newcommand{\R}{\mathbb{R}}

\newcommand{\abs}[1]{\left\lvert#1\right\rvert}

\newcommand{\set}[1]{\left\{#1\right\}}
\newcommand{\dom}[1]{\text{dom}\left(#1\right)}

\newcommand*{\QEDB}{\hfill\ensuremath{\square}}%

\newcommand{\Tr}[1]{\ensuremath{\text{Tr}}\left(#1\right)}

\newcommand*\tageq{\refstepcounter{equation}\tag{\theequation}}
\allowdisplaybreaks



\begin{document}
\begin{frontmatter}

\title{Accelerated Continuous-Time Approximate Dynamic Programming via Data-Assisted Hybrid Control\tnoteref{t1,t2}}
\journal{~}
\author{Daniel E. Ochoa \corref{cor1}}
\ead{daniel.ochoa@colorado.edu}

\author{Jorge I. Poveda}
\ead{jorge.poveda@colorado.edu}

\affiliation{organization={Department of Electrical, Energy and Computer Engineering. University of Colorado Boulder},
            city={Boulder},
            postcode={80305}, 
            state={Colorado},
            country={USA}}
            
\tnotetext[t1]{Research supported in part by NSF grant number CNS-1947613.}
\cortext[cor1]{Corresponding Author.}

\begin{abstract}                
We introduce a new closed-loop architecture for the online solution of approximate optimal control problems in the context of continuous-time systems. Specifically, we introduce the first algorithm that incorporates dynamic momentum in actor-critic structures to control continuous-time dynamic plants with an affine structure in the input. By incorporating dynamic momentum in our algorithm, we are able to accelerate the convergence properties of the closed-loop system, achieving superior transient performance compared to traditional gradient-descent based techniques. In addition, by leveraging the existence of past recorded data with sufficiently rich information properties, we dispense with the persistence of excitation condition traditionally imposed on the regressors of the critic and the actor. Given that our continuous-time momentum-based dynamics also incorporate periodic discrete-time resets that emulate restarting techniques used in the machine learning literature, we leverage tools from hybrid dynamical systems theory to establish asymptotic stability properties for the closed-loop system.  We illustrate our results with a numerical example.
%
\end{abstract}


\end{frontmatter}
\section{Introduction}
Recent technological advances in computation and sensing have incentivized the development and implementation of data-assisted feedback control techniques previously deemed intractable due to their computational complexity. Among these techniques, reinforcement learning (RL) has emerged as a practically viable tool with remarkable degrees of success in robotics \citep{ibarz2021train}, autonomous driving \citep{kiran2021deep}, water-distribution systems \citep{martinez2020multi}, among other cyber-physical applications, see \citep{vamvoudakis2021handbook}. These types of algorithms, are part of a large landscape of adaptive systems that aim to control a plant while simultaneously optimizing a performance index in a model-free way, with closed-loop stability guarantees.

In this paper, we focus on a particular class of infinite horizon RL problems from the perspective of approximate optimal control and approximate adaptive dynamic programming (AADP). Specifically, we study the optimal control problem for nonlinear continuous-time and control-affine deterministic plants, interconnected with approximate adaptive optimal controllers \citep{kamalapurkar2018reinforcement} in an actor-critic configuration. These types of adaptive controllers aim to find, in real time, the solution to the Hamilton-Jacobi-Bellman (HJB) equation by measuring the output of the nonlinear dynamical system while making use of two approximation structures:
\begin{itemize}
    \item a critic, used to estimate the optimal value function of the optimal control problem, and
    \item an actor, used to estimate the optimal feedback controller.
\end{itemize}
Our goal is to design online adaptive dynamics for the real-time tuning of the aforementioned structures, while simultaneously achieving closed-loop stability and high transient performance. To achieve this, and motivated by the widespread usage of momentum-based gradient dynamics in practical RL settings \citep{mnih2015human}, we study continuous-time actor-critic dynamics inspired by a class of ordinary differential equations (ODEs) that can be seen as continuous-time counterparts of Nesterov's accelerated optimization algorithm \citep{ODE_Nesterov}. Such types of algorithms have gained popularity in optimization and related fields due to the fact that they can minimize smooth convex functions at a rate of order $\mathcal{O}(1/t^2)$ \citep{Wibisono1}. The main source for the acceleration property in these ODEs comes from the addition of momentum to gradient-based dynamics, in conjunction with a vanishing dynamic damping coefficient. However, as recently shown in \citep{zero_order_poveda_Lina} and \citep{poveda2020heavy}, the non-uniform convergence properties that emerge in these types of dynamics complicates their use in feedback systems with plant dynamics in the loop. In this paper, we overcome these challenges by incorporating resets into the proposed momentum-based algorithms, similar to restarting heuristics studied in the machine learning literature, see \citep{Candes_Restarting} and \citep{ODE_Nesterov}. Our resulting actor-critic controller is naturally modeled by a hybrid dynamical system that incorporates continuous-time and discrete-time dynamics, which we analyze using tools from \citep{bookHDS}.

A traditional assumption in the literature of continuous-time actor-critic RL is that the regressors used in the parameterizations satisfy a persistence of excitation condition along the trajectories of the plant. However, in practice, this condition can be difficult to verify a priori. To circumvent this issue, in this paper we consider a data-assisted approach, where a finite amount of past ``sufficiently rich'' recorded data is used to guarantee asymptotic learning in the closed-loop system. As a consequence, the resulting data-assisted hybrid control algorithm concurrently uses real-time and recorded data, similar in spirit to concurrent-learning (CL) techniques \citep{chowdhary2010concurrent}.
By using Lyapunov-based tools for hybrid dynamical systems, we analyze the interconnection of an actor-critic neural-network (NN) controller and the nonlinear plant, establishing that the trajectories of the closed-loop system remain ultimately bounded around the origin of the plant and the optimal actor and critic NN parameters. Since the resulting closed-loop system has suitable regularity properties in terms of continuity of the dynamics, our stability results are in fact robust with respect to arbitrarily small additive disturbances that can be adversarial in nature, or that can arise due to numerical implementations. To the best knowledge of the authors, these are the first theoretical stability guarantees of continuous-time accelerated actor-critic algorithms for neural network-based adaptive dynamic programming controllers in nonlinear deterministic settings.

The rest of this paper is organized as follows: Section \ref{sec:prelim} presents the notation and some concepts on hybrid dynamical systems, Section \ref{sec:problem} presents the problem statement and some preliminaries on optimal control. Section \ref{sec:critic} introduces the hybrid momentum-based dynamics for the update of the critic NN, Section \ref{sec:actor} presents the update dynamics for the actor NN, and Section \ref{sec:stability:closedloop} studies the properties of closed-loop system. In Section \ref{sec:numerical} we study a numerical example illustrating our theoretical results. 
\section{Preliminaries}\label{sec:prelim}
\textbf{Notation:}  We denote the real numbers by $\R$, and we use $\R_{\ge 0}\subset \R$ to denote the non-negative real line. We use $\R^n$ to represent the $n$-dimensional Euclidean space and $\abs{\cdot}$ to denote its usual vector norm. Given $A\in\R^{n\times n}$, we use $\abs{A}$ to denote the induced 2-norm for matrices, and we infer its distinction with the vector norm depending on the context. We use $\Tr{A}$ to denote the trace operator on matrices. Given a compact set $\mathcal{A}\subset\R^n$ and a vector $z\in\R^n$, we use $|z|_{\mathcal{A}}\coloneqq \min_{s\in\mathcal{A}}\abs{z-s}$ to represent the minimum distance of $z$ to $\mathcal{A}$. We also use $r\mathbb{B}$ to denote a closed ball in the Euclidean space, of radius $r>0$, and centered at the origin. We use $I_n\in\R^{n\times n}$ to denote the identity matrix, and $(x,y)$ for the concatenation of the vectors $x$ and $y$, i.e., $(x,y)\coloneqq [x^\top,y^\top]^\top$. A function $\gamma:\R_{\ge 0}\to \R_{\ge 0}$ is said to be of class-$\mathcal{K}$ ($\gamma\in\mathcal{K}$), if it is continuous, zero at zero, and nondecreasing.  A function $\beta:\R_{\geq0}\times\R_{\geq0}\to\R_{\geq0}$ is said to be of class-$\mathcal{K}\mathcal{L}$ ($\beta\in\mathcal{KL}$) if $\beta(\cdot,s)\in\mathcal{K}$ for each $s\in\R_{\geq0}$, it is non-increasing in its second argument, and  $\lim_{s\to\infty}\beta(r,s)=0$ for each $r\in\R_{\geq0}$. The gradient of a real valued function $f:\R^n\to\R$ is defined as a column vector and denoted by $\nabla f$. For a vector valued function $g:\R^n\to\R^m$, we use $\diffp{g(x)}{x}\in\R^{m\times n}$ to denote its Jacobian matrix.\\

\noindent\textbf{Hybrid Dynamical Systems:}  To study our algorithms, we will use tools from hybrid dynamical systems (HDS) theory \citep{bookHDS}. A HDS with state $x\in\mathbb{R}^n$, has dynamics
\begin{equation}\label{HDS}
x\in C,~\dot{x}=F(x),~~~\text{and}~~~x\in D,~~x^+= G(x),
\end{equation}
where $F:\R^n\to\R^n$ is called the \emph{flow map}, $G:\R^n\to\R^n$ is called the \emph{jump map}, and $C\subset\R^n$ and $D\subset\R^n$ are closed sets, called the \emph{flow set} and the \emph{jump set}, respectively. We use $\mathcal{H}=(C,F,D,G)$ to denote the elements of the HDS $\mathcal{H}$. Solutions $x:\text{dom}(x)\to\R^n$ to system \eqref{HDS} are indexed by a continuous-time parameter $t$, which increases continuously during flows, and a discrete-time index $j$, which increases by one during jumps. Thus, the notation $\dot{x}$ in \eqref{HDS} represents the derivative $\frac{dx(t,j)}{dt}$; and $x^+$ in \eqref{HDS} represents the value of $x$ after an instantaneous jump, i.e., $x(t,j+1)$. Therefore, solutions $x:\text{dom}(x)\to\mathbb{R}^n$ to system \eqref{HDS} are defined on \emph{hybrid time domains}. For a precise definition of hybrid time domains and solutions to HDS of the form \eqref{HDS}, we refer the reader to \citep[Ch.2]{bookHDS}. The following definitions will be instrumental to study the stability and convergence properties of systems of the form \eqref{HDS}.
\vspace{0.1cm}
\begin{definition}\label{UAS}
The compact set $\mathcal{A}\subset C\cup D$ is said to be \emph{uniformly asymptotically stable} (UAS) for system \eqref{HDS} if $\exists$ $\beta\in\mathcal{K}\mathcal{L}$ and $r>0$ such that every solution $x$ with $x(0,0)\in r\B\cap (C\cup D)$ satisfies:  
\begin{equation}\label{KLbound}
|x(t,j)|_{\mathcal{A}}\leq \beta(|x(0,0)|_{\mathcal{A}},t+j),~ \forall~(t,j)\in\text{dom}(x). 
\end{equation}
When $\beta(r,s)=c_1re^{-c_2s}$ for some $c_1,c_2>0$, the set $\mathcal{A}$ is said to be \emph{uniformly exponentially stable} (UES). \QEDB
\end{definition}
%
\section{Problem Statement}\label{sec:problem}
%
Consider a control-affine nonlinear dynamical plant
\begin{equation}\label{system:dynamics}
    \dot{x} = f(x) + g(x)u,
\end{equation}
where $x\in\R^n$ is the state of the system, $u\in U\subset \R^m$ is the input, and $f:\R^n\to\R^n$ and $g:\R^n\to\R^{n\times m}$ are locally Lipschitz functions. Our goal is to design a stable algorithm able to find --in real time-- a control law $u^*$ that minimizes the cost functional $V:\R^n\times \mathcal{U}_V\to \R$ given by:
\begin{equation}\label{cost_functional}
V(x_0, u) \coloneqq \int_{0}^\infty r\Big(x\big(\tau\big),u\left(x(\tau)\right)\Big)d\tau,
\end{equation}
where $x\big(t\big)$ represents a solution to \eqref{system:dynamics} from the initial condition $x(0) = x_0$, that results from implementing a feedback law $u$, belonging to a class of admissible control laws $\mathcal{U}_V$ characterized as follows:  
\begin{definition}%
\citep[Definition 1]{beard1997galerkin} Given the dynamical system in \eqref{system:dynamics}, a feedback control $u:\R^n\to \R^m$ is \emph{admissible} with respect to the cost functional $V$ in \eqref{cost_functional} if

\begin{itemize}
    \item $u$ is continuous,
    \item $u$ renders system \eqref{system:dynamics} UAS,
    \item $V(x_0,u)<\infty$ for all $x_0\in \R^n$.\QEDB
\end{itemize}
We denote the set of \emph{admissible feedback laws} as $\mathcal{U}_V$.
\end{definition}
\noindent In \eqref{cost_functional}, we consider cost functions $r:\R^n\times \R^m\to \R$ of the form
    $r(x,u) \coloneqq Q(x) + R(u),$
where the state-cost is given by $Q(x) \coloneqq x^\top \Pi_x x$ with $\Pi_x\succ 0$, and the control-cost is given by $R(u)\coloneqq u^\top \Pi_u u$ with $\Pi_u\succ 0$.
%
%
To find the optimal control law that minimizes \eqref{cost_functional}, we study the \emph{Hamiltonian function} $H:\R^n\times \R^m\times \R^n\to \R$ related to  \eqref{system:dynamics} and \eqref{cost_functional}, given by
\begin{align}
H(x, u, \nabla V) 
                  &\coloneqq \nabla V^\top (f(x) + g(x)u) + Q(x) + R(u).\tageq{\label{hamiltonian}}
\end{align}
Using \eqref{hamiltonian}, a necessary optimality condition for $u^*$ is given by Pontryagin's maximum principle \citep{liberzon2011calculus}:
\begin{align*}
 u^*(x) &= \argmin_{u\in\mathcal{U}_V}    H(x, u, \nabla V^*)
        \implies u^*(x) = -\frac{1}{2}\Pi_u^{-1}g(x)^\top \nabla V^*(x),\tageq{\label{optimal:control:law}}
\end{align*}
where $V^*$ represents the \emph{optimal value function}:
\begin{equation*}
    V^*(x)\coloneqq \inf_{u\in \mathcal{U}_V} V(x,u(\cdot))
\end{equation*}
On the other hand, under the assumption that $V^*$ is continuously differentiable, the optimal value function can be shown to satisfy the Hamilton-Jacobi-Bellman equation \citep[Ch. 1.4]{kamalapurkar2018reinforcement}:
\begin{equation*}
    \diffp{V^*}{t} = -H(x,u^*, \nabla V^*)\quad \forall x\in\R^n.
\end{equation*}
Since the functional in \eqref{cost_functional} does not have an explicit dependence on $t$, it follows that $\diffp{V^*}{t}=0$, and hence $H(x,u^*, \nabla V^*) =0$, meaning that for all $x\in \R^n$, the following holds:
\begin{align*}
    \nabla V^{*^\top} \big(f(x) + g(x)u^*(x)\big) + Q(x) + R\Big(u^*(x)\Big)=0.\tageq{\label{hjb}}
\end{align*}
The time-invariant Hamilton-Jacobi-Bellman equation in \eqref{hjb}, allows for a state-dependent characterization of optimality. Therefore, by using the optimal control law in \eqref{optimal:control:law}, and assuming that the system dynamics \eqref{system:dynamics} are known, the form \eqref{hjb} could be leveraged to find $V^*$. Unfortunately, finding an explicit closed-form expression for $V^*$, and thus for the optimal control law, is, in general, an intractable problem. However, the utility of \eqref{hjb} is not completely lost. As we shall show in the following sections, online and historical ``measurements'' of \eqref{hjb} can be leveraged in real time to estimate the optimal control law $u^*$ while concurrently rendering a neighborhood of the origin of system \eqref{system:dynamics} asymptotically stable.
\section{Data-Assisted Critic Dynamics}\label{sec:critic}
To leverage the form of \eqref{hjb}, we consider the following parameterization of the optimal value function $V^*(x)$:
\begin{equation}\label{critic:approximation}
    V^*(x) = \theta_c^{*^\top}\phi_c(x) + \epsilon_c(x)\quad\forall x\in K,
\end{equation}
where $K\subset \R^n$ is a compact set, $\theta_c^*\in \R^{l_c}$, $\phi_c:\R^n\to \R^{l_c}$ is a vector of continuously differentiable basis functions, and $\epsilon_c:\R^n\to \R$ is the approximation error. The parameterization \eqref{critic:approximation} is always possible on compact sets due to the continuity properties of $V$ and the universal approximation theorem \citep{hornik1990universal}. This parametrization results in an optimal Hamiltonian of the form $H_p^*\coloneqq H(x,u^*,\diffp{\phi_c}{x}^\top \theta_c^* + \nabla\epsilon_c)$ given by:
\begin{align}
    H_p^*(x) &= \theta_c^{*^\top}\psi(x,u^*(x)) +  Q(x) + R\left(u^*(x)\right)  + \nabla\epsilon_c(x)^\top \left(f(x) + g(x)u^*(x)\right)\label{optimal:parametrized:hamiltonian},
\end{align}
where we defined $\psi:\R^n\times\R^m\to \R^{l_c}$ as:
\begin{equation}\label{definition:psi}
    \psi(x,u)\coloneqq \diffp{\phi_c(x)}{x}\left(f(x) + g(x)u\right).
\end{equation}
We note that the explicit dependence of $\psi:\R^n\times \R^m\to\R^{l_c}$ on the control action $u$, defined in \eqref{definition:psi}, is a fundamental departure from the previous approaches studied in the context of concurrent learning (CL) NN actor-critic controllers, such as those considered in \citep{vamvoudakis2015asymptotically} and \citep{kamalapurkar2016model}. In particular, we note that in the context of CL the data used to estimate the optimal value function $V^*$ is generated from measurements of the optimal Hamiltonian which, by definition, incorporates the optimal control law $u^*$. Hence, the need to include $u$ as part of the regressor vectors $\psi$ becomes crucial; this dependence characterizes how far our recorded measurements of a Hamiltonian are from the optimal Hamiltonian $H_p^*$. Indeed, this distance will explicitly emerge in our convergence and stability analysis. Naturally, the dependence of \eqref{definition:psi} on $u$ will impose stronger conditions on the recorded data needed to estimate $V^*$.  
%
%
%

\noindent Assuming we have access to $\phi_c$, we can define a \emph{critic} neural network as:
\begin{equation}\label{critic:nn}
    \hat{V}(x) \coloneqq \theta_c^\top\phi_c(x),~~\forall x\in K,
\end{equation}
which will serve as an approximation of the optimal value function $V^*$ in \eqref{critic:approximation}. This critic NN results in an estimated Hamiltonian:
\begin{equation}\label{estimated:optimal:hamiltonian}
    H\left(x,u,\nabla\hat{V}\right) \coloneqq \theta_c^\top\psi\left(x,u\right) + Q(x) + R(u),
\end{equation}
which we will use to design the update dynamics of the critic parameters $\theta_c$. In particular, our goal is to use previously recorded data from trajectories of the plant to ensure asymptotic stability of the set of optimal critic parameters $\set{\theta_c^*}$, while simultaneously enabling the incorporation of instantaneous measurements from the plant. Towards this end, we will assume enough ``richness'' properties in the recorded data, a notion that is captured by a relaxed (and finite-time) version of \emph{persistence of excitation} (PE); see \citep{chowdhary2010concurrent} and \citep{Astrom:Book}. 

\begin{assumption}\label{assumption:richness}
Let $\{\psi\left(x_k,u^*(x_k)\right)\}_{k=1}^{N}$ be a sequence of recorded data, and define:
\begin{align}
    \Lambda &\coloneqq \sum_{k=1}^N\Psi(x_k,u^*(x_k))\Psi(x_k,u^*(x_k))^\top,\quad
    \Psi(x,u) \coloneqq \frac{\psi(x,u)}{1+\psi(x,u)^\top\psi(x,u)}.\tageq{\label{lambda:def}}
\end{align}
There exists $\underline{\lambda} \in \R_{>0}$  such that $\Lambda \succeq \underline{\lambda} I_{n}$, i.e., the data is $\underline{\lambda}$-sufficiently-rich ($\underline{\lambda}$-SR). \QEDB
\end{assumption}
\begin{remark}
In this paper, we study reinforcement learning dynamics that do not make explicit usage of exploration signals with standard PE properties, which can be difficult to guarantee in practice. Instead, we assume access to samples obtained by observing the action of optimal values $u^*(x_k)$ acting on the plant. Note however that this does not imply knowledge of the optimal control policy as a whole, but only of a finite number of \emph{demonstrations} from an ``expert'' policy. Similar requirements commonly arise in the literature of imitation learning, or inverse reinforcement learning, and have been recently shown in practice to reduce the exploratory requirements of online reinforcement learning algorithms, with mild assumptions in the sampling of the demonstrations. For recent discussions on these topics in the discrete-time stochastic reinforcement learning setting we refer the reader to \citep{ciosek2022imitation} and \citep{rashidinejad2021bridging}.
\end{remark}
\noindent Now, we consider the instantaneous and data-dependent errors of the estimated Hamiltonian with respect to the optimal one:
\begin{align*}
    e^i\left(\theta_c,x,u\right) &\coloneqq H\left(x, u, \nabla \hat{V}\right) - H\left(x, u^*(x), \nabla V^*\right)\\
    %
    %
    &~= \theta_c^\top\psi\left(x,u\right) + Q(x) + R\left(u\right),
    \\
    e^d_{k}(\theta_c) &\coloneqq H\left(x_k, u^*(x_k), \nabla \hat{V}\right) - H\left(x_k, u^*(x_k), \nabla V^*\right)\\
    &~= \theta_c^\top\psi\left(x_k,u^*(x_k)\right) + Q(x_k) + R\left(u^*(x_k)\right),
\end{align*}
where we used the fact that $H\left(x, u^*(x), \nabla V^*\right) = 0$. Moreover, we define the \emph{joint instantaneous and data-dependent} error as:
\begin{align*}
    e\left(\theta_c,x,u\right) &\coloneqq \frac{1}{2}\Bigg(\rho_i\frac{e^i\left(x,\theta_c,u\right)^2}{\left(1+\abs{\psi(x,u)}^2\right)^2}+\rho_d\sum_{k=1}^N\frac{e^d_k(\theta_c)^2}{\left(1+\abs{\psi\left(x_k,u^*(x_k)\right)}^2\right)^2}\Bigg)\tageq{\label{joint_error}},
\end{align*}
where $\rho_i\in\R_{\ge 0}$ and $\rho_d\in \R_{>0}$ are tunable gains. Since we are interested in designing real-time training dynamics for the estimation of the optimal parameters $\theta_c^*$, we compute the the gradient of \eqref{joint_error} with respect to $\theta_c$ as follows:
\begin{align}
	    \nabla_{\theta_c} e(\theta_c,x,u) &= \rho_i \left(\rule{0cm}{0.75cm}\right.\Psi(x,u)\Psi(x,u)^\top\theta_c+\frac{\psi(x,u)\left[Q(x) + R(u)\right]}{\left(1+\psi(x,u)^\top\psi(x,u)\right)^2} \left.\rule{0cm}{0.75cm}\right)\notag\\
                  &\qquad+
                  \rho_d\left(\Lambda\theta_c + \sum_{k=1}^N\frac{\psi(x_k,u^*(x_k))\left[Q(x_k) + R\left(u^*(x_k)\right)\right]}{\left(1+\psi\left(x_k,u^*(x_k)\right)^\top\psi\left(x_k,u^*(x_k)\right)\right)^2}\right),\label{gradient:error}
\end{align}
where $\Lambda$ and $\Psi$ are defined in Assumption \ref{assumption:richness}.\\ 
The ``propagated'' error to the HJB equation that results from the approximate parametrization of $V^*$ in \eqref{critic:approximation},  is given by:
\begin{align*}
    \epsilon_{\text{HJB}}(x) &\coloneqq H(x,u^*(x),\nabla V^*) - H\left(x,u^*,\diffp{\phi_c(x)}{x}^\top \theta_c^*\right)\\
    &=-\nabla\epsilon_c^\top(x) \Big(f(x) + g(x)u^*(x)\Big).\tageq{\label{errorHJB}}
\end{align*}
The following assumption is standard, and it is satisfied when the involved functions are continuous and $K$ is compact.
\begin{assumption}\label{assumption:bounds:critic}
There exist $\overline{\phi_c},~\overline{d\phi_c},~\overline{\epsilon_{c}},~\overline{d\epsilon_{c}},~\overline{\epsilon_{\text{HJB}}},\overline{g} \in \R_{>0}$ such that
\begin{align*}
 &\abs{\phi_c(x)}\leq \overline{\phi_c},~~\abs{\diffp{\phi_c(x)}{x}}\leq \overline{d\phi_c},~~\abs{\epsilon_c(x)}\leq\overline{\epsilon_{c}},\\
 &\abs{\nabla\epsilon_{c}(x)}\leq \overline{d\epsilon_{c}},~\abs{\epsilon_{
 HJB}(x)}\leq \overline{\epsilon_{HJB}},~\abs{g(x)}\leq \overline{g}\quad\forall x\in K,
\end{align*} 
where $K$ is the same set considered in \eqref{critic:approximation}.\QEDB
\end{assumption}
\subsection{Critic Dynamics via Data-Driven Hybrid Momentum-Based Control}
To design fast asymptotically stable dynamics for the estimate $\theta_c$, we propose a new class of momentum-based critic dynamics inspired by accelerated gradient flows with restarting mechanisms, such as those studied in \citep{ODE_Nesterov} and \citep{Candes_Restarting}. Specifically, we consider the following hybrid dynamics of the form \eqref{HDS}, with state $y\coloneqq(\theta_c, p, \tau)$ and elements: 
%
\begin{subequations}\label{critic:dynamics}
\begin{align}
    &C_c\coloneqq \set{y\in \R^{2l_c + 1}~:~\tau \in [T_0,T]},\qquad F_c(y,x,u) \coloneqq \begin{pmatrix}
    \frac{2}{\tau}\left(p - \theta_c\right)\\
    -2k_c\nabla_{\theta_c} e(\theta_c,x,u)\\
    \frac{1}{2}
    \end{pmatrix},\label{critic:flow}\\
    & D_c\coloneqq \set{y\in \R^{2l_c + 1}~:~\tau =T },\qquad\qquad G_c(y) \coloneqq \begin{pmatrix}
    \theta_c\\
    \theta_c\\
    T_0
    \end{pmatrix},\label{jumpmap1}
\end{align}
\end{subequations}
where $k_c\in\R_{>0}$ is a tunable gain, and $(p,\tau)$ are auxiliary states that are periodically reset every time $\tau=T$ via the jump map \eqref{jumpmap1}, with $\infty>T>T_0>0$. The dynamical system in \eqref{critic:dynamics} flows in continuous time according to \eqref{critic:flow} whenever the timer variable $\tau$ is in $[T_0,T]$. As soon as $\tau$ hits $T$, the algorithm \eqref{critic:dynamics} resets the timer variable to $T_0$, as well as the momentum variable $p$ to $\theta_c$, while leaving $\theta_c$ unaffected. Accordingly, after the first reset, the system exhibits periodic resets every $\Delta T = 2(T-T_0)$ intervals of time. The following assumption provides data-dependent tuning guidelines for the resetting frequency of the timer variable $\tau$, which will be leveraged in our stability results.
\begin{assumption}\label{tunable:parameters:critic}
The tunable parameters $(T_0,T,k_c,\rho_i,\rho_d)$ satisfy $2\rho_d\underline{\lambda}>\rho_i$ and
\begin{equation}\label{parameters:ineq}
T_0^2 + \frac{1}{2k_c\underline{\lambda}\rho_d}<T^2<\frac{8\rho_d\underline{\lambda}}{k_c \rho_i^2},    
\end{equation}
where $\underline{\lambda}$ is the level of richness of the recorderd data defined in Assumption \ref{assumption:richness}.\QEDB
\end{assumption}
For system \eqref{critic:dynamics}, we study stability properties with respect to the compact set:
\begin{subequations}\label{Ac:definition}
\begin{align}
        \mathcal{A}_c &\coloneqq \mathcal{A}_{\theta_c,p}\times [T_0, T],\\
            \mathcal{A}_{\theta_c,p} &\coloneqq  \set{(\theta_c,p)\in\mathbb{R}^{2l_c}:p_c=\theta_c,~\theta_c=\theta_c^*}.
\end{align}
\end{subequations}
The following theorem is the first main result of this paper. All the proofs are presented in the Appendices.
\begin{theorem}\label{thm:iss:critic}
Given a number $l_c$ of basis functions $\phi_c$ parametrizing the critic NN, and a compact set $K\subset \R^n$, suppose that Assumptions \ref{assumption:richness}, \ref{assumption:bounds:critic} and \ref{tunable:parameters:critic} are satisfied. Then, there exists $(\kappa,c)\in\R_{>0}\times \R_{>0}$ and class-$\mathcal{K}_\infty$ functions $\gamma_1$ and $\gamma_2$, such that for every solution $y=(\theta_c,p, \tau)$ to \eqref{critic:dynamics} with initial condition $y(0,0)=(\theta_c(0,0),p(0,0),\tau(0,0))$, and using the control policy $u(\cdot)\in \mathcal{U}_V$ on the plant, the critic parameters $\theta_c$ satisfy 
\begin{align}
    \abs{\theta_c(t,j)-\theta_c^*} &\leq \kappa e^{-c (t+j) }\abs{y(0,0)}_{\mathcal{A}_{c}}+ \gamma_2\left(\abs{\tilde{u}(x(t,j))}\right) + \gamma_1(\overline{\epsilon_{\text{HJB}}}), \tageq{\label{thm:iss:critic:eq}}
\end{align}
where $\tilde{u}(x(t,j)): = u(x(t,j))- u^*(x(t,j))$, for all $(t,j)\in\dom{y}$ \QEDB
\end{theorem}
The presence of a residual optimal-control mismatch term in \eqref{thm:iss:critic:eq} of the form $\gamma_2(\abs{u(x)-u^*(x)})$, represents a crucial difference with respect to previous CL adaptive dynamic approaches, such as those studied in \citep{vamvoudakis2015asymptotically} and \citep[Ch. 4 ]{kamalapurkar2018reinforcement}. This term is a direct byproduct of our definition of $\psi$ in \eqref{definition:psi}, its dependence on the control action $u$, and its appearance in the error gradient \eqref{gradient:error}. In principle, the emergence of this term in Theorem 1 is agnostic to the particular gradient-based update dynamics for the critic NN, regardless of the inclusion or not of momentum. Since $\gamma_2\in\mathcal{K}$, the larger the difference between the nominal input $u$ and the optimal feedback law $u^*$, the greater the residual error in the convergence of $\theta_c$. In particular, the bound \eqref{thm:iss:critic:eq} describes a semi-global practical input-to-state stability property that, to the best knowledge of the authors, is novel in the context of CL-based RL. In the next section we will show that the residual error $\gamma_2(|\tilde{u}|)$ can be removed by incorporating an additional actor NN in the system.
\begin{remark}
In contrast to standard data-driven gradient-descent dynamics for the estimation of the optimal value function $V^*$, which can achieve exponential rates of convergence proportional to $\underline{\lambda}$ (cf. \citep{kamalapurkar2016model, chowdhary2010concurrent}), under the assumptions of Theorem \ref{thm:iss:critic} the critic update dynamics \eqref{critic:dynamics} can achieve exponential convergence with rates proportional to  $\sqrt{\underline{\lambda}}$. As shown in \citep{zero_order_poveda_Lina}, momentum-based dynamics of this form can achieve these rates using the restarting parameter
\begin{equation}
T=T^{*}:=e\sqrt{\frac{1}{2k_c\rho_d\underline{\lambda}}+ T_0^2}. 
\end{equation}
This property is particularly useful in settings where the level of richness of the data-set is limited, i.e., when $\underline{\lambda}\ll 1$, which is common in practical applications.
\end{remark}
Theorem \ref{thm:iss:critic} guarantees exponential convergence to a neighborhood of the optimal parameters $\set{\theta_c^*}$ that define the optimal value function $V^*$.
Consequently, by continuity, and on compact sets, $\hat{V}$ would converge to an $\epsilon$-approximation of $V^*$, which can be leveraged by the control law \eqref{optimal:control:law} to stabilize system \eqref{system:dynamics}. However, as noted in \citep{doya2000reinforcement}, implementing only critic structures for the control of nonlinear dynamical systems of the form \eqref{system:dynamics} can lead to poor closed-loop transient performance. To tackle this issue, we consider an auxiliary dynamical system, called the \emph{actor}, which will serve as an estimator of the optimal controller that acts on the plant. 
%
\begin{figure}
    \centering
    \includegraphics[width=0.75\linewidth]{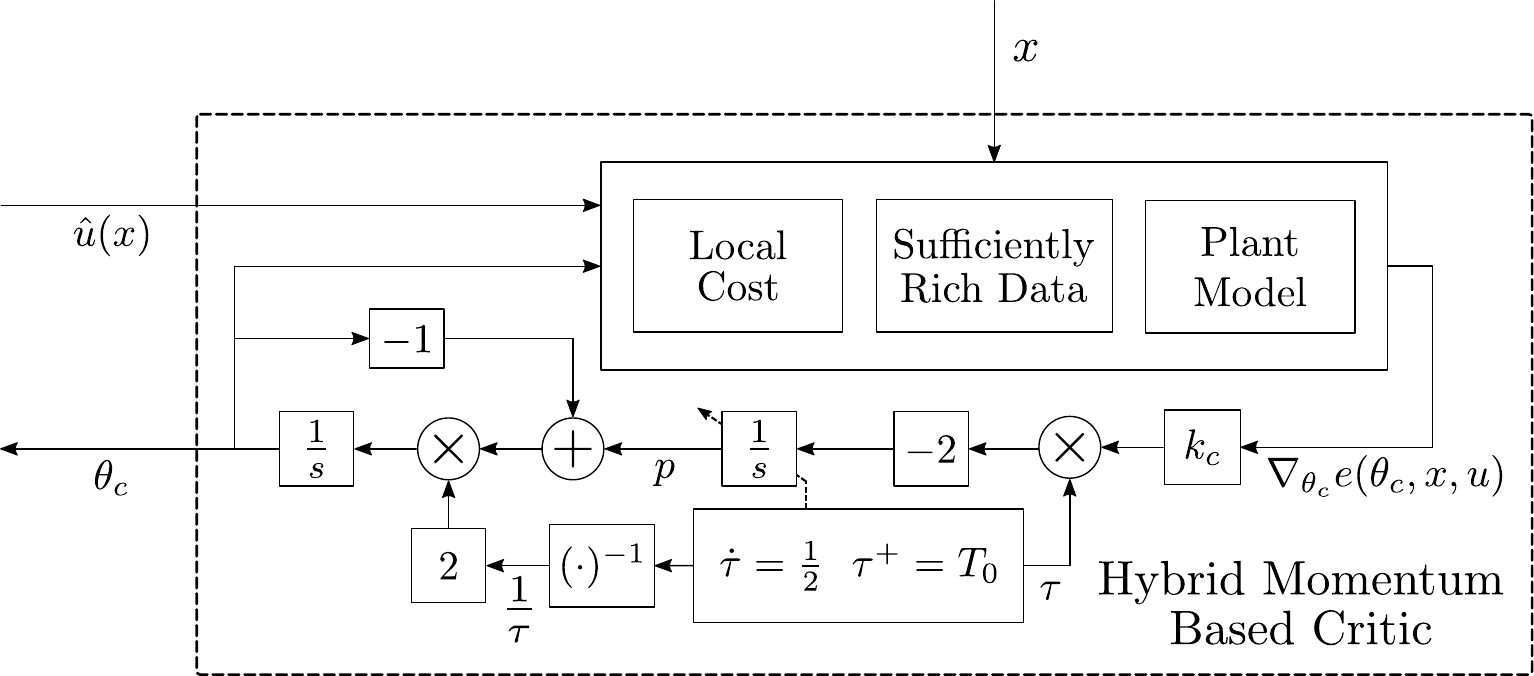}
    \caption{Proposed Hybrid Momentum Based Dynamics for the training of the Critic subsystem}
    \label{fig:critic}
\end{figure}
\section{Actor Dynamics}\label{sec:actor}
\begin{figure}
     \centering
     \includegraphics[width=0.55\linewidth]{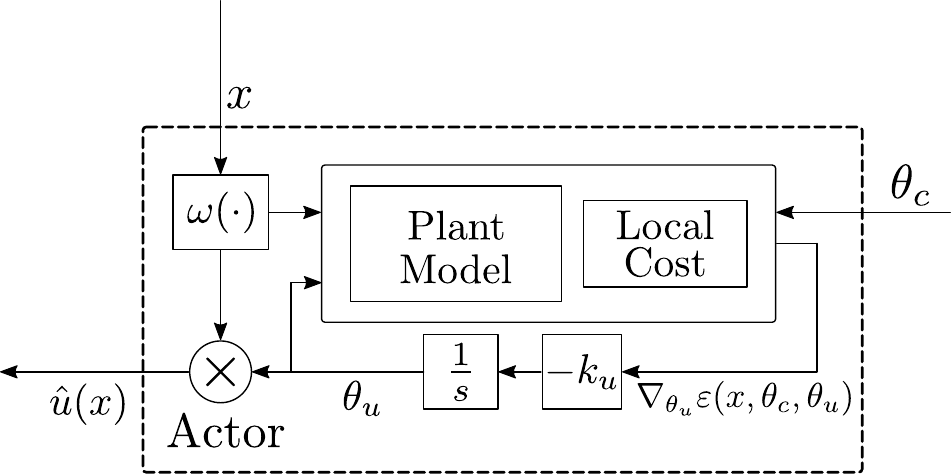}
     \caption{Actor Subsystem}
     \label{fig:actorSubsystem}
\end{figure}
Using the optimal value parametrization described in Section \ref{sec:critic} the optimal control law can written as:
\begin{equation}
    u^*(x) = -\frac{1}{2}\Pi_u^{-1}g(x)^\top \left[\diffp{\phi_c(x)}{x}^\top\theta_c^* + \nabla\epsilon_c(x)\right], \quad \forall x\in K.
\end{equation}
Therefore, using $\diffp{\phi_c(x)}{x}$ and $g(x)$ we can implement an actor neural-network given by:
\begin{equation}\label{feedback:law}
    \hat{u}(x) =  \omega(x)^\top\theta_u,
\end{equation}
where $\omega:\R^n\to \R^{l_c\times m}$ is defined as:
\begin{equation}
    \omega(x) \coloneqq -\frac{1}{2}\diffp{\phi_c(x)}{x}g(x)\Pi_u^{-1}.
\end{equation}
To guarantee convergence of $\hat{u}$ to $u^*$, we design update dynamics for $\theta_u\in\R^{l_c}$ based on the minimization of the error:
\begin{align*}
    \varepsilon(x,\theta_c,\theta_u) &\coloneqq \frac{1}{2}\Bigg[\alpha_1\frac{\varepsilon_a(x,\theta_c,\theta_u)^\top\varepsilon_a(x,\theta_c,\theta_u)}{1+\Tr{\omega(x)^\top\omega(x)}}+\alpha_2\varepsilon_b(\theta_c,\theta_u)^\top\varepsilon_b(\theta_c,\theta_u)\Bigg],\\
    \varepsilon_a(x,\theta_c,\theta_u) &\coloneqq \hat{u}(x)- \omega(x)^\top\theta_c= \omega(x)^\top\left(\theta_u-\theta_c\right),\\
    \varepsilon_b(\theta_c,\theta_u) &\coloneqq \theta_u-\theta_c,\tageq{\label{actor:update:alt}}
\end{align*}
which satisfies:
\begin{align*}
    \nabla_{\theta_u}\varepsilon(x,\theta_c,\theta_u)
                                     %
                                     %
                                     %
                                     &= \Omega(x)(\theta_u-\theta_c),
\end{align*}
where
\begin{align*}
    \Omega(x) &\coloneqq \alpha_1 \frac{\omega(x)\omega(x)^\top}{1+\Tr{\omega(x)^\top\omega(x)}} + \alpha_2 I \in \R^{l_c\times l_c}\quad\forall x\in \R^n.\tageq{\label{actor:Omega:definition}}
\end{align*}
Based on these definitions, we consider the following gradient-descent dynamics for the actor neural-network:
\begin{align*}
    \dot{\theta}_u   = F_u(\theta_u, x,\theta_c) \coloneqq -k_u \nabla_{\theta_u}\varepsilon(x,\theta_c,\theta_u),\tageq{\label{actor:dynamics:alt}}
\end{align*}
where $k_u\in\R_{>0}$ is a tunable gain. A scheme representing these update dynamics is shown in Figure \ref{fig:actorSubsystem}.
%
\section{Momentum-Based Actor-Critic Feedback System}\label{sec:stability:closedloop}
Consider the closed-loop resulting from the interconnection between the plant \eqref{system:dynamics}, the critic update dynamics \eqref{critic:dynamics}, the actor update dynamics \eqref{actor:dynamics:alt} and the feedback law in \eqref{feedback:law} shown in Figure \ref{fig:closedLoop}, and given by:
\begin{subequations}\label{closed:loop:dynamics}
\begin{align}
    \dot{x} &= f(x) + g(x)\hat{u}(x),\qquad x^+ = x,\\
    \dot{y} &= F_c(y, x, \hat{u}(x)), \qquad\quad~ y^+ = G_c(y),\\
    \dot{\theta}_u &= F_u(\theta_u, x, \theta_c),\qquad\qquad \theta_u^+ = \theta_u,
\end{align}
\end{subequations}
and with flow set and jump set given by $C=\R^n\times C_c\times \R^{l_c}$ and $D=\R^n\times D_c\times \R^{l_c}$ respectively, where $C_c$ and $D_c$ are as defined in \eqref{critic:dynamics}. Let $z\coloneqq (x,y,\theta_u)$ be the overall state of the closed-loop system, and define:
\begin{align*}
    \mathcal{A} &\coloneqq \set{0}\times \mathcal{A}_c\times \set{\theta_c^*}.
\end{align*}
The following is the main result of this paper.
\begin{theorem}\label{thm:main:result}
Given the vector of basis functions $\phi_c:\R^n\to\R^{l_c}$ parametrizing the critic NN and a compact set $K_z\coloneqq K\times K_y\times K_\theta\subset \R^n\times \R^{2l_c + 1}\times \R^{l_c}$, where $K$ is given as in \eqref{critic:approximation}, suppose that Assumption \ref{assumption:richness}-\ref{tunable:parameters:critic} are satisfied. Then, there exists $\beta\in\mathcal{KL}$, $\gamma\in\mathcal{K}$ and tunable parameters $(\rho_i, \rho_d, k_c, k_u, \alpha_1, \alpha_2)$, such that for every solution $z=(x, y, \theta_u)$ to the closed-loop system \eqref{closed:loop:dynamics}, with initial condition $z(0,0)=(x(0,0),y(0,0),\theta_u(0,0))\in K_z$, there exists $\tilde{T}>0$ such that for all $(t,j)\in\text{dom}(z)$:
\begin{align*}
    &\abs{z(t,j)}_\mathcal{A} \leq \beta(\abs{z(0,0)}_\mathcal{A}, t+j)+ \gamma(\abs{\left(\overline{\epsilon_{\text{HJB}}},\overline{d\epsilon_c}\right)}) + \nu,
\end{align*}    
for all $0\leq t+j\leq \tilde{T}$, and
\begin{align*}
    \abs{z(t,j)}_\mathcal{A}\leq \gamma(\abs{\left(\overline{\epsilon_{\text{HJB}}},\overline{d\epsilon_c}\right)}) + \nu,~~~\forall~\tilde{T}\leq t+j,
\end{align*}
for some $\nu>0$ constant. \QEDB
\end{theorem}
 Theorem \ref{thm:main:result} establishes asymptotic convergence to a neighborhood of the compact set $\mathcal{A}$ as $\left(\overline{\epsilon_{\text{HJB}}},\overline{d\epsilon_c}\right)\to 0$ from any compact set $K_z$ modulo some error $\nu$, under a suitable choice of tunable parameters. To the best knowledge of the authors this is the first result providing stability certificates for continuous-time actor-critic reinforcement learning using recorded data and accelerated value-function estimation dynamics with momentum. In addition, since the resulting closed-loop system in \eqref{closed:loop:dynamics} is given by a well-posed hybrid system, the stability results are robust with respect to arbitrarily small additive disturbances on the states and dynamics \citep[Ch. 7]{bookHDS}.
\section{Numerical Example}\label{sec:numerical}
\begin{figure*}[t]
     \centering
    \begin{subfigure}[c]{0.33\linewidth}
        \centering
        \includegraphics[width=0.75\textwidth]{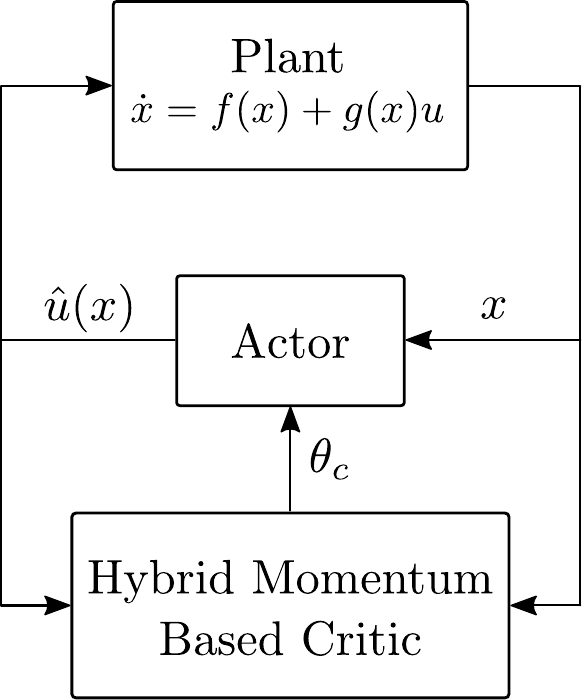}
        \caption{Closed-Loop System}    \label{fig:closedLoop}
    \end{subfigure}\hfill
    \begin{subfigure}[c]{0.66\linewidth}
        \centering
        \includegraphics[width=\textwidth]{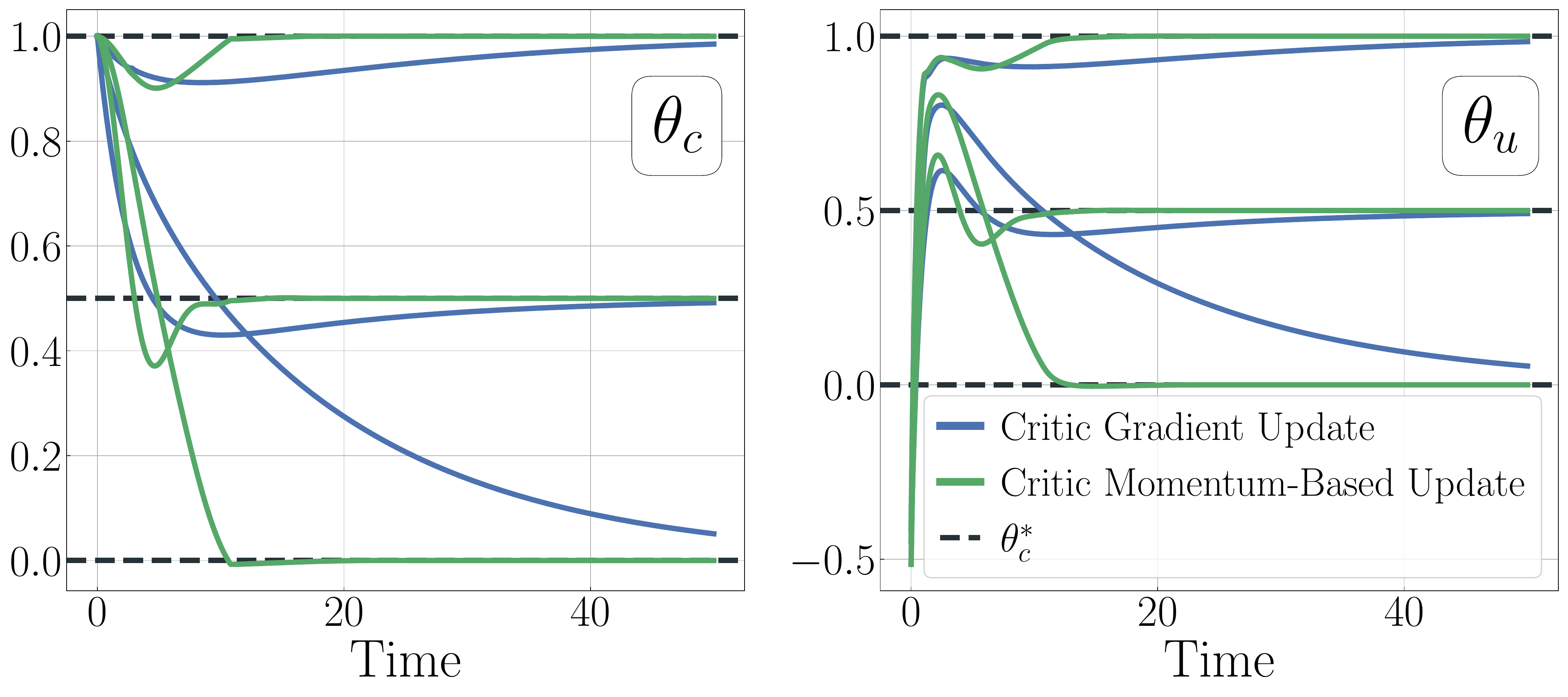}
        \caption{Convergence of the critic (left) and actor (right) neural networks' weights to the optimal values.}    \label{fig:numericalResults}
    \end{subfigure}
    \caption{Closed-Loop System Diagram and Numerical Example}
    \label{fig:numericalExamples}
\end{figure*}
In this section, we present a numerical experiment that illustrates our theoretical results. In particular, we study the following nonlinear control-affine plant:
\begin{subequations}\label{numerical:example:diffeq}
\begin{align}
    &\dot{x} = f(x) + g(x)u,\\
    &f(x) =\begin{bmatrix}-x_1 + x_2\\-\frac{1}{2}\Bigg(x_1 - x_2\Big(1-\cos(2x_1+2)^2\Big)\Bigg)\end{bmatrix},\\
    &g(x)\coloneqq 
     \begin{bmatrix}0\\ \cos(2x_1) + 2\end{bmatrix},
\end{align}
\end{subequations}
with local state and control costs given by $Q(x) = x^\top x$ and $R(u)=u^2$ \citep{kamalapurkar2016model}. The optimal value function for this setting is given by $V^*(x)=\frac{1}{2}x_1^2 +x_2^2$ with optimal control law given by $u^*(x) = -(\cos(2x_1)+ 2)x_2$. Using this information, we choose $\phi_c(x) = (x_1^2, x_1x_2, x_2^2)$, and we implement the prescribed hybrid momentum-based dynamics in \eqref{critic:dynamics} for the update of the critic neural network, and the update dynamics for the actor described in \eqref{actor:dynamics:alt}. We obtain the results shown in Figure \ref{fig:numericalResults} with $x(0,0)=(-10,10)$, $\theta_c(0,0)=(1,1,1)$ and $\theta_u\in [0,1]^3$. We compare the results with the case in which the critic neural-network is updated with the gradient-descent dynamics of \citep{vamvoudakis2015asymptotically}, and where the sufficiently rich data is a set of $16$ data points obtained by sampling the dynamics \eqref{numerical:example:diffeq} in a grid around the origin of size $4\times 4$. In our simulations we use $T_0 = 0.1, T = 5.5$ for the momentum-based dynamics in \eqref{critic:dynamics}. These particular values are obtained by using the level of richness $\underline{\lambda}$ of the data-set, and the inequalities in \eqref{parameters:ineq} in order to ensure compliance with Assumption \ref{tunable:parameters:critic}. For both reinforcement learning dynamics we use $k_c=1, k_u=1, \rho_d=1$ and $\rho_i=1$. As shown in the figure both update dynamics are able to converge to $\set{\theta_c^*}$, with $\theta_c^*=(1/2, 0 ,1)$  describing the optimal value function $V^*$. However, the hybrid-based dynamics are able to significantly improve the transient performance of the learning mechanism.\footnote{The code used to implement this simulation can be found in the following repository: https://github.com/deot95/Accelerated-Continuous-Time-Approximate-Dynamic-Programming-through-Data-Assisted-Hybrid-Control}
\section{Conclusions}
In this paper, we introduced the first stability guarantees for deterministic continuous-time actor-critic reinforcement learning with accelerated training of neural network structures. To do so, we studied a novel hybrid momentum-based estimation dynamical system for the critic NN, which estimates, in real time, the optimal value function. Our stability analysis leveraged the existence of rich recorded data taken from a finite number of samples along optimal trajectories and inputs of the system. We showed that this finite sequence of samples can be used to train the controller to achieve online optimal performance with fast transient performance. Closed-loop stability was established using tools from hybrid dynamical systems theory. Potential extensions include the study of similar accelerated training dynamics for the actor subsystem, as well as considering reinforcement learning problems in hybrid plants.

\bibliographystyle{ieeetr}
\bibliography{references}

\appendix
\section{Proof Theorem \ref{thm:iss:critic}}
\subsection{Gradient of Critic Error-Function in Deviation Variables}
First, using \eqref{errorHJB} together with $H(x,u^*(x),\nabla V^*)=0$ for all $x$, we obtain:
\begin{equation}\label{appproximate:hjb}
    \psi(x,u^*(x))^\top\theta_c^*+ Q(x) + R\left(u^*(x)\right) = \epsilon_{\text{HJB}}(x) .
\end{equation}
Thus, using \eqref{gradient:error} and \eqref{appproximate:hjb}, we can rewrite the gradient of $e(\theta_c, x,u )$ as follows:
\begin{align*}
    \nabla_{\theta_c} e(\theta_c, x, u)
                  &= \Theta(x,u)\left(\theta_c - \theta_c^*\right)+ \upsilon_\epsilon(x,u) + \chi(x,u), \tageq{\label{gradient:parameter:error}}
\end{align*}
where
\begin{align}
 \Theta(x,u) &\coloneqq \rho_i\Psi(x,u)\Psi(x,u)^\top + \rho_d\Lambda, \label{definition:Theta}
\end{align}
and
\begin{align}
    \upsilon_\epsilon(x,u)&\coloneqq \rho_i\frac{\psi(x,u)\epsilon_{\text{HJB}}(x)}{\left(1+\abs{\psi(x,u)}^2\right)^2}
    +\rho_d\sum_{k=1}^N\frac{\psi(x_k,u^*(x_k))\epsilon_{\text{HJB}}(x_k)}{\left(1+\abs{\psi(x_k,u^*(x_k))}^2\right)^2}~\in\R^{l_c},\label{definition:upsilon:critic}\\
    \chi(x,u)&\coloneqq  \frac{\rho_i\psi(x,u)\left[\diffp{\phi_c(x)}{x}g(x)\left(u-u^*(x)\right)\right]^\top\theta^*_c}{\left(1+\abs{\psi(x,u)}^2\right)^2}+\frac{\rho_i\psi(x,u)\left[R(u)-R(u^*(x))\right]}{\left(1+\abs{\psi(x,u)}^2\right)^2} ~\in\R^{l_c},
\end{align}
which, by using the fact that $\frac{r}{\left(1+r^2\right)^2}\leq \frac{3\sqrt{3}}{16},\forall r\in \R_{\ge 0}$, satisfy:
\begin{subequations}\label{upper:bounds:critic}
\begin{align}
    \abs{\upsilon_\epsilon(x,u)}&\leq \frac{3\sqrt{3}}{16}\overline{\epsilon_{\text{HJB}}}\left(\rho_i + N\rho_d\right),\label{upper:bound:inputerrorCritic}\\
    \abs{\chi(x,u)}&\leq \rho_i\frac{3\sqrt{3}}{16}\Bigg(\overline{g}\left(\overline{d\phi_c}\left[1+\abs{\theta_c^*}\right]+\overline{d\epsilon_c}\right)\abs{u-u^*(x)} + \lambda_{\max}\left(\Pi_u\right)\abs{u-u^*(x)}^2 \Bigg) \label{upper:bound:mismatchError:Alt}.
\end{align}
\end{subequations}
The following Lemma will be instrumental for our results.
\begin{lemma}\label{lemma:theta}
If the data is $\underline{\lambda}$-sufficiently-rich, then there exist  $\overline{\Theta},\underline{\Theta}\in \R_{>0}$ such that
\begin{align*}
\underline{\Theta} I_n \preceq \Theta(x,u) \preceq \overline{\Theta} I_n\qquad \forall x\in \R^n,~\forall u \in \R^m.
\end{align*}
\end{lemma}
\begin{proof}
Let $\theta\in \R^{l_c}$ be arbitrary. Since, by assumption, the data is $\underline{\lambda}$-SR it follows that:
\begin{align*}
    \theta^\top \Theta(x,u)\theta &= \theta^\top \rho_i\Psi(x,u)\Psi(x,u)^\top\theta + \theta^\top \rho_d \Lambda \theta\\
    %
    %
    %
    &\ge \rho_d \underline{\lambda} \abs{\theta}^2\\
    &\implies \Theta(x,u) \succeq \underline{\Theta} I_{l_c},~\forall (x,u)\in \R^n\times\R^m,\tageq{\label{lower:bound:Theta}}
\end{align*}
where $\underline{\Theta} \coloneqq \rho_d\underline{\lambda}$. On the other hand, using the fact that $\abs{aa^\top}=\abs{a}^2,~\forall a\in\R^n$, we obtain that:
\begin{equation*}
    \abs{\Psi(x,u)\Psi(x,u)^\top} = \abs{\Psi(x,u)}^2 \leq 1,~\forall (x,u)\in\R^n\times \R^m,
\end{equation*}
we obtain:
\begin{align*}
    \theta^\top \Theta(x,u)\theta &= \theta^\top \rho_i\psi(x,u)\psi(x,u)^\top\theta + \theta^\top \rho_d \Lambda \theta\\
                            %
                            %
                            %
                            %
                            %
                            &\leq \left(\rho_i +  \rho_d\lambda_{\max}\left(\Lambda\right)\right)\abs{\theta}^2\\
                            &\implies \Theta(x,u) \preceq \overline{\Theta}I_{l_c},\quad \forall (x,u)\in \R^n\times \R^m,
\end{align*}
where $\overline{\Theta}\coloneqq \rho_i +  \rho_d\lambda_{\max}\left(\Lambda\right)$.
\end{proof}
\subsection{Lyapunov-Based Analysis}
Recall from Section \ref{sec:critic} that $y=(\theta_c,p,\tau)$, suppose that the assumptions of Theorem \ref{thm:iss:critic} hold and consider the Lyapunov candidate function $V_c:\R^{l_c}\times\R^{l_c}\times \R_{> 0}\to\R_{\ge 0}$ given by:
\begin{align*}
    V_c(y)&\coloneqq \frac{\abs{p-\theta_c}^2}{4} + \frac{\abs{p-\theta_c^*}^2}{4} + k_c\rho_d\tau^2\frac{\left(\theta_c-\theta_c^*\right)\top \Lambda\left(\theta_c-\theta_c^*\right)}{2},\tageq{\label{Vc:def}}
\end{align*}
where $\Lambda$ was defined in Assumption \ref{assumption:richness} and which satisfies:
\begin{align*}
    \underline{c}\abs{y}_{\mathcal{A}_c}^2 \leq V_c&(y) \leq \overline{c}\abs{y}_{\mathcal{A}_c}^2,\tageq{\label{Vc:quadraticBounds}} \\
\underline{c}\coloneqq \min\set{\frac{1}{4},\frac{k_c\rho_d T_0^2\underline{\lambda}}{2}},&\quad  \overline{c}\coloneqq \set{\frac{3}{4},~\frac{1}{2}\left(1{+}k_c\rho_d T^2\overline{\lambda}\right)},
\end{align*}
where $\overline{\lambda}\coloneqq \lambda_{\max}\left(\Lambda\right)$. Now, let $u\in\mathcal{U}_V$, and consider the time derivative of $V_c$ along the continuous-time evolution of the critic subsystem, i.e., $\dot{V}_c = \nabla_y V_c(y)^\top \dot{y}$. Then, by using \eqref{gradient:parameter:error} and Lemma \ref{lemma:theta}, and some algebraic manipulation, $\dot{V}_c$ can be shown to satisfy
\begin{align*}
    \dot{V}_c &\leq -\begin{pmatrix}
        \abs{p-\theta_c}& \abs{\theta_c-\theta_c^*}
    \end{pmatrix} M(\tau)\begin{pmatrix}
        \abs{p-\theta_c}\\ \abs{\theta_c-\theta_c^*}
    \end{pmatrix} + 2\sqrt{2}k_c y_{\mathcal{A}_c} \big(\abs{\upsilon_\epsilon(x)} + \abs{\chi(x,u(x))}\big),\tageq{\label{critic:vdot:alt2:alt}}
\end{align*}
where
\begin{equation}\label{matrix:worstCase}
    M(\tau)\coloneqq \begin{pmatrix}
        \frac{2}{k_c\tau^2}& -\frac{\rho_i}{2}\\ -\frac{\rho_i}{2}& \underline{\Theta}
    \end{pmatrix},
\end{equation}
and $\mathcal{A}_c$ was defined in Section \ref{sec:critic}. Since $2\rho_d\underline{\lambda} > \rho_i$ and $T^2< \frac{8\rho_d\underline{\lambda}}{k_c\rho_i^2}$ by means of Asssumption \ref{assumption:bounds:critic}, and $\tau(t,j)\in[T_0,T],~\forall (t,j)\in \dom{y}$ by construction of the critic update dynamics \eqref{critic:dynamics}, it follows that $M(\tau)\succeq \underline{r}$ with $\underline{r}\coloneqq \underline{\Theta} - \frac{\rho_i}{2}$. Hence, from \eqref{critic:vdot:alt2:alt} and using \eqref{upper:bounds:critic}, we obtain that:
\begin{align*}
    \dot{V}_c &\leq -\underline{r}\abs{y}_{\mathcal{A}_c}^2 + \abs{y}_{\mathcal{A}_c}\Big(\gamma_\nu\left(\overline{\epsilon_{\text{HJB}}}\right) + \gamma_\chi\left(\abs{u(x)-u^*(x)}\right)\Big),\tageq{\label{dotVc:prefinal}}
\end{align*}
where $\gamma_\nu,\gamma_\chi \in \mathcal{K}_\infty$ are given by:
\begin{align*}
    \gamma_\nu(r) &\coloneqq \frac{3\sqrt{6}}{8}\left(\rho_i + N\rho_d\right)r,\quad \gamma_\chi(r)\coloneqq c_\chi(r+r^2),\\
    c_\chi &\coloneqq \frac{3\sqrt{6}}{8}\rho_i\max\set{\overline{g}\left(\overline{d\phi_c}\left[1+\abs{\theta_c^*}\right]+\overline{d\epsilon_c}\right),~ \lambda_{\max}\left(\Pi_u\right)}.
\end{align*}
Thus, letting $d_c\in (0,1)$, and using \eqref{Vc:quadraticBounds},  \eqref{dotVc:prefinal}:
\begin{subequations}\label{dotVc}
\begin{align}
    \dot{V}_c &\leq -\frac{\underline{r}(1-d_c)}{\overline{c}}V_c(y),\quad
    \forall \abs{y}_{\mathcal{A}_c} \ge \frac{1}{d_c}\Big(\gamma_\nu\left(\overline{\epsilon_{\text{HJB}}}\right) + \gamma_\chi\left(\abs{u(x)-u^*(x)}\right)\Big).
\end{align}
\end{subequations}
On the other hand, the change of $V_c$ during the jumps in the update dynamics for the critic \eqref{critic:dynamics}, satisfies:
\begin{equation}\label{DeltaVc}
V_c\left(y^+\right) -  V_c(y) \leq -\eta V_c(y),
\end{equation}
with $\eta\coloneqq 1-\frac{T_0^2}{T^2}-\frac{1}{2k_c\rho_d\underline{\lambda}T^2}$ which satisfies $\eta\in(0,1)$ by means of Assumption \ref{assumption:bounds:critic}. Together, \eqref{dotVc} and \eqref{DeltaVc}, in conjuction with the quadratic bounds of \eqref{Vc:quadraticBounds}, imply the results of Theorem \ref{thm:iss:critic} via \citep[Prop 2.7]{cai2009characterizations} and the fact that $\abs{\theta_c(t,j)-\theta^*_c}\leq \abs{y(t,j)}_{\mathcal{A}_c}\leq \abs{(\theta_c(t,j),p(t,j))}_{\mathcal{A}_{\theta_c,p}}$ for all $(t,j)\in\dom{y}$.\qed
\section{Proof of Theorem \ref{thm:main:result}}
\subsection{Gradient of Actor Error-Function in Deviation Variables}
First, note that we we can write \eqref{actor:update:alt} as:
\begin{align*}
    \nabla_{\theta_u} \varepsilon_a(x,\theta_c,\theta_u) &=  \Omega(x)\left(\theta_u - \theta_c^* -\left(\theta_c-\theta_c^*\right)\right),
\end{align*}
and consider the following Lemma, instrumental for our results.
\begin{lemma}\label{lemma:omega}
There exists $\overline{\Omega},\underline{\Omega}\in \R_{>0}$ such that 
\begin{equation*}
    \underline{\Omega} I_{l_c} \preceq \Omega(x) \preceq \overline{\Omega}I_{l_c}.
\end{equation*}
\end{lemma}
\begin{proof}
Let $\theta\in \R^{l_c}$ be arbitrary. Then, by the definition of $\Omega:\R^n\to \R^{l_c\times l_c}$ in \eqref{actor:Omega:definition}, it follows that:
\begin{align*}
    \theta^\top \Omega(x) \theta 
    = \alpha_1\frac{\abs{\omega(x)^\top \theta}^2}{1+\Tr{\omega(x)^\top\omega(x)}} + \alpha_2\abs{\theta}^2
    \ge \alpha_2\abs{\theta}^2
    &\implies \Omega(x) \succeq \underline{\Omega} I_{l_c},\quad \forall x\in \R^n,
\end{align*}
where $\underline{\Omega}\coloneqq \alpha_2$. On the other hand, we obtain:
\begin{align*}
    \theta^\top \Omega(x)\theta 
    &= \left(\alpha_1\frac{\abs{\omega(x)}^2}{1+\abs{\omega(x)}_F^2} + \alpha_2\right)\abs{\theta}^2
    \leq \overline{\Omega}\abs{\theta}^2
    \implies \Omega(x) \preceq \overline{\Omega}I_{l_c}, \quad \forall x\in \R^n,
\end{align*}
where $\overline{\Omega} \coloneqq \alpha_1 + \alpha_2$, $\abs{A}_F$ represents the Frobenius norm and where we used $\abs{A}\leq \abs{A}_F,~ \forall A\in\R^{l_c\times l_c}$ and $\frac{r^2}{1+r^2}\leq 1~\forall r\in \R$.
\end{proof}
Now, consider the Lyapunov function
\begin{subequations}
\begin{align}
    & \mathcal{V}(z) \coloneqq V_o(x) + V_c(y) + V_a(\theta_u),\label{fullV:def}\\
    & V_o(x) \coloneqq V^*(x),\quad 
    V_a(\theta_u)\coloneqq \frac{1}{2}\abs{\theta_u-\theta_u^*}^2,\label{VoVa:def}
\end{align}
\end{subequations}
where $V_c$ was defined in \eqref{Vc:def} and where we recall that $z=(x,y,\theta_u)$. By \citep[Lemma 4.3]{khalil}, and since $V_o=V^*$ is a continuous and positive definite function in $\R^n$, there exist $\underline{\gamma}_o,\overline{\gamma}_o\in \mathcal{K}$ such that $\underline{\gamma}_o\left(\abs{x}\right)\leq V_o(x) \leq \overline{\gamma}_o(\abs{x})$. Hence, using \eqref{Vc:quadraticBounds}, and the fact that sum of class $\mathcal{K}$ is in turn of class $\mathcal{K}$, there exist $\underline{\gamma}_{\mathcal{V}},\overline{\gamma}_{\mathcal{V}}\in\mathcal{K}$ such that:
\begin{align}
    \underline{\gamma}_{\mathcal{V}}\left(\abs{z}_{\mathcal{A}}\right)\leq \mathcal{V}(z) \leq  \overline{\gamma}_{\mathcal{V}}\left(\abs{z}_{\mathcal{A}}\right)\label{fullV:bounds}
\end{align}
Now, the time derivative of $\dot{V}_o=\nabla V_o(x)^\top \dot{x}$ along the trajectories of \eqref{closed:loop:dynamics} satisfies:
\begin{align*}
 \dot{V}_o  & \leq -Q(x)  + \frac{\overline{g}^2\lambda_{\max}\left(\Pi_u^{-1}\right)}{2}\left(\overline{d\phi}_c
 \abs{\theta^*_c}+ \overline{d\epsilon_c}\right)\left(\overline{d\phi_c}\abs{\theta_u-\theta_c^*} + \overline{d\epsilon_c}\right)\tageq{\label{dotv1:alt}}.
\end{align*}
On the other hand, making use of Lemma \ref{lemma:omega}, for the time derivative of $\dot{V}_a=\nabla_{\theta_u} V_a(\theta_u)^\top\theta_u$ we obtain:
\begin{align*}
    \dot{V}_a &\leq -k_u\alpha_2\abs{\theta_u-\theta_c^*}^2 + k_u \overline{\Omega}\abs{\theta_u-\theta_c^*}\abs{\theta_c-\theta_c^*}.
    \tageq{\label{dotv3:final:alt}}
\end{align*}
Hence, using \eqref{critic:vdot:alt2:alt}, \eqref{dotv1:alt}, and \eqref{dotv3:final:alt}, together with the upper bounds in \eqref{upper:bounds:critic}, we obtain that the time derivative of $\mathcal{V}$ along the trajectories of the closed-loop system satisfies:
\begin{align*}
    \dot{\mathcal{V}}&\leq -Q(x) - \underline{r}\abs{y}_{\mathcal{A}_c}^2 - k_u\alpha_2\abs{\theta_u - \theta_c^*}^2\\
    &\quad + c_y\abs{y}_{\mathcal{A}_c} + c_u\abs{\theta_u - \theta_c^*} + c_{yu}\abs{\theta_u - \theta_c^*}\abs{y}_{\mathcal{A}_c} \\
    &\qquad+ c_{yu^2}\abs{y}_{\mathcal{A}_c}\abs{\theta_u - \theta_c^*}^2
    + c_0,\tageq{\label{lyapunov:closed:loop:partial:local:alt}}
\end{align*}
where
\begin{align*}
c_y &\coloneqq \frac{3\sqrt{6}}{8}k_c\Bigg(\overline{\epsilon_{\text{HJB}}}\left(\rho_i + N\rho_d\right) +\frac{1}{2}\overline{g}^2\rho_i\Bigg[\lambda_{\max}\left(\Pi_u^{-1}\right)\left(\overline{d\phi_c}\left[1+\abs{\theta_c^*}\right]+\overline{d\epsilon_c}\right)\overline{d\epsilon_c}+ \lambda_{\max}\left(\Pi_u\right)\lambda_{\max}\left(\Pi_u^{-1}\right)^{2}\overline{d\epsilon_c}^2\Bigg]\Bigg),\\
c_u &\coloneqq \frac{1}{2}\left(\overline{d\phi}_c \abs{\theta^*_c}+ \overline{d\epsilon_c}\right)\overline{g}^2\lambda_{\max}\left(\Pi_u^{-1}\right)\overline{d\phi_c},\\
c_{yu}&\coloneqq \frac{3\sqrt{6}k_c}{16}\Bigg(2k_u\overline{\Omega}+\overline{g}^2\rho_i\lambda_{\max}\left(\Pi_u^{-1}\right)\left(\overline{d\phi_c}\left[1+\abs{\theta_c^*}\right]+\overline{d\epsilon_c}\right)\overline{d\phi_c} \Bigg),\\
c_{yu^2}&\coloneqq \frac{3\sqrt{6}}{16}k_c \overline{g}^2\rho_i\lambda_{\max}\left(\Pi_u\right)\lambda_{\max}\left(\Pi_u^{-1}\right)^2\overline{d\phi_c}^2,\\
c_0 &\coloneqq \frac{1}{2}\left(\overline{d\phi}_c \abs{\theta^*_c}+ \overline{d\epsilon_c}\right)\overline{g}^2\lambda_{\max}\left(\Pi_u^{-1}\right)\overline{d\epsilon_c}.
\end{align*}
Then, for all $\abs{\theta_u - \theta_c^*}\leq \frac{c_{yu}}{c_{yu^2}}$, by using $Q(x)=x^\top\Pi_x x$ and letting $d_1\in (0,1)$, from \eqref{lyapunov:closed:loop:partial:local:alt}, $\dot{\mathcal{V}}$ can be further upper bounded as:
\begin{align*}
    \dot{\mathcal{V}} &\leq -\lambda_{\min}\left(\Pi_x\right)\abs{x}^2 -(1-d_1)\left(\underline{r}\abs{y}_{\mathcal{A}_c}^2+k_u\alpha_2\abs{\theta_u - \theta_c^*}^2\right)\\
    &\quad +c_y\abs{y}_{\mathcal{A}_c} + c_u \abs{\theta_u - \theta_c^*}+ c_0\\
    &\qquad -\begin{pmatrix}
        \abs{y}_{\mathcal{A}_c} & \abs{\theta_u - \theta_c^*}
    \end{pmatrix}\begin{pmatrix}
        d_1\underline{r} & -c_{yu}\\
        -c_{yu}  & d_1k_u\alpha_2
    \end{pmatrix}\begin{pmatrix}
        \abs{y}_{\mathcal{A}_c} \\ \abs{\theta_u - \theta_c^*}
    \end{pmatrix}.\tageq{\label{dotV:quadraticDominance:alt}}
\end{align*}
Now, pick a set of tunable parameters $(\rho_i,\rho_d,k_c,k_u)$ such that $\underline{r}\ge \frac{c_{yu}^2}{d_1^2k_u\alpha_2}$ so that from  \eqref{dotV:quadraticDominance:alt}, we obtain:
\begin{subequations}\label{dotfullV}
\begin{align}
    &\qquad\quad\dot{\mathcal{V}}\leq -(1-d_2)d_z\abs{z}_{\mathcal{A}}^2,\quad\forall \abs{z}_{\mathcal{A}} \ge \max\set{\frac{c_0}{2d_{yu}},~\frac{2d_{yu}}{d_2d_z}} ,~\abs{\theta_u - \theta_c^*}\leq \frac{c_{yu}}{c_{yu^2}},
\end{align}
\end{subequations}
with
\begin{align*}
d_z &\coloneqq  \min\set{\lambda_{\min}\left(\Pi_x\right),~(1-d_1)\underline{r},~k_u\alpha_2},\\
d_{yu}&\coloneqq \max\set{2c_{yu},~c_0},\quad d_2\in(0,1).
\end{align*}
Notice that for every compact set $K_\theta$ of initial conditions for $\theta_u$ we can pick suitable $\rho_i,\rho_d,\alpha_1,\alpha_2,k_c,k_u$ to satisfy $K_\theta\subset\frac{c_{yu}}{c{yu^2}}\B$ such that \eqref{dotfullV} holds for every trajectory with $\theta_u(0,0)\in K_{\theta}$.
Now, during jumps $x$ and $\theta_u$ do not change, and hence $\mathcal{V}$ satisfies:
\begin{equation}\label{DeltafullV}
    \mathcal{V}(z^+) - \mathcal{V}(z) = V_c(y^+) -  V_c(y) \leq -\eta V_c(y).
\end{equation}
The result of the theorem follows by using the strong-decrease of $\mathcal{V}$ during flows outside a neighborhood of $\mathcal{A}$ described in \eqref{dotfullV}, the non-increase of $\mathcal{V}$ during jumps given in \eqref{DeltafullV}, by noting that, by design, the closed-loop dynamics are a well-posed HDS which experiences periodic jumps followed by intervals of flow of length $T-T_0>0$ (c.f. \citep{ochoa2021accelerated}), and by following the same arguments of \citep[Prop 3.27]{bookHDS} and \citep[Prop. 2.7]{cai2009characterizations}.\qed

\end{document}